\documentclass[12pt, reqno]{amsart}
\usepackage{amsmath, amsthm, amscd, amsfonts, amssymb, graphicx, color}
\usepackage[bookmarksnumbered, colorlinks, plainpages]{hyperref}
\input{mathrsfs.sty}

\newtheorem{theorem}{Theorem}[section]
\newtheorem{lemma}[theorem]{Lemma}

\newtheorem{corollary}[theorem]{Corollary}
\theoremstyle{definition}
\newtheorem{definition}[theorem]{Definition}

\theoremstyle{remark}
\newtheorem{remark}[theorem]{Remark}

\numberwithin{equation}{section}

\begin{document}

\title[The Hyers--Ulam stability and the Moore--Penrose inverse]{Relationship between the Hyers--Ulam stability and the
Moore--Penrose inverse}
\author[Q. Huang, M.S. Moslehian]{Qianglian Huang$^{1*}$ and Mohammad Sal Moslehian$^2$}
\address{$^1$ College of Mathematics, Yangzhou University,
Yangzhou 225002, China; School of Mathematical Sciences, Monash
University, VIC 3800, Australia} \email{qlhmath@yahoo.com.cn}
\address{$^2$ Department of Pure Mathematics, Ferdowsi University of
Mashhad, P.O. Box 1159, Mashhad 91775, Iran.}
\email{moslehian@um.ac.ir and moslehian@memebr.ams.org}
\urladdr{\url{http://profsite.um.ac.ir/~moslehian/}}

\subjclass[2010]{47A55, 46A32, 39B82, 47A05, 47A30.}

\keywords{Hyers--Ulam stability; Moore--Penrose inverse; generalized
inverse; reduced minimum modulus; closed linear operator;
$T-$boundedness; semi-Fredholm operator.}

\begin{abstract}
In this paper, we establish a link between the Hyers--Ulam stability
and the Moore--Penrose inverse, that is, a closed operator has the
Hyers--Ulam stability if and only if it has a bounded Moore--Penrose
inverse. Meanwhile, the stability constant can be determined in
terms of the Moore--Penrose inverse. Based on this result, some
conditions for the perturbed operators having the Hyers--Ulam
stability are obtained and the Hyers--Ulam stability constant is
expressed explicitly in the case of closed operators. In the case of
the bounded linear operators we obtain some characterizations for
the Hyers--Ulam stability constants to be continuous. As an
application, we give a characterization for the Hyers--Ulam
stability constants of the semi-Fredholm operators to be continuous.
\end{abstract}
\maketitle


\section{Introduction and Preliminaries}

More than a half century ago, Ulam \cite{ULA} proposed the first
stability problem concerning group homomorphisms, which was
partially solved by Hyers \cite{HYE} in the framework of Banach
spaces. Later, Aoki \cite{AOK} proved the stability of the additive
mapping and Th.M. Rassias \cite{TRAS} investigated the stability of
the linear mapping for mappings $f$ when the norm of the Cauchy
difference $f(x+y) - f(x) - f(y)$ is bounded by the expression
$\varepsilon(\|x\|\sp p+\|y\|\sp p)$ for some $\varepsilon\ge 0$ and
some $ 0 \leq p < 1 $. J.M. Rassias \cite{JRAS} considered the same
problem with $\varepsilon(\|x\|\sp p\,\|y\|\sp p)$. A large number
of papers have been published in connection with various
generalizations of Hyers--Ulam theorem in several wide frameworks.
In particular, it is nearly related to the notion of perturbation
\cite{CIMS, Met} and geometry of Banach spaces \cite{DM, TX}. The
interested reader is referred to books \cite{CZE, H-I-R, JUN, MOS}
and references therein.

In 2003, Miura, Miyajima and Takahasi \cite{13, 14} investigated the
notion of the Hyers--Ulam stability of a mapping between normed
linear spaces and obtained some stability results for particular
linear differential operators. Takagi, Miura and Takahasi \cite{19}
considered the Hyers--Ulam stability of bounded linear operators in
Banach spaces. After then, Hirasawa and Miura \cite{5} gave some
necessary and sufficient conditions under which a closed operator in
a Hilbert space has the Hyers--Ulam stability. Moslehian and Sadeghi
\cite{M-S} studied the Hyers--Ulam stability of $T$-bounded
operators. They also discussed the best constant of Hyers--Ulam
stability. In the sequel, we need some terminology.
\begin{definition}\label{Def1.1} Let $X,Y$ be normed
linear spaces and let $T$ be a (not necessarily linear) mapping from
$X$ into $Y$. We say that $T$ has the Hyers--Ulam stability if there
exists a constant $K>0$ with the property: (a) For any $y$ in the
range $R(T)$ of $T$, $\varepsilon>0$ and $x$ in the Domain $D(T)$ of
$T$ with $\|Tx-y\|\leq\varepsilon$, there exists an element $x_0\in
D(T)$ such that $Tx_0=y$ and $\|x-x_0\|\leq K\varepsilon$. We call
such $K>0$ a Hyers--Ulam stability constant for $T$ and denote by
$K_T$ the infimum of all Hyers--Ulam stability constants for $T$. If
$K_T$ is a Hyers--Ulam stability constant for $T$, then $K_T$ is
called the Hyers--Ulam stability constant for $T$.
\end{definition}
Roughly speaking, if $T$ has the Hyers--Ulam stability, then to each
$\varepsilon$-approximate solution $x$ of the equation $Tx=y$ there
corresponds an exact solution $x_0$ of the equation in a
$K_\varepsilon$-neighborhood of $x$; see \cite{5}.

 \begin{remark} \cite{5, M-S}
If $T$ is linear then condition (a) is equivalent to:\\
(b) For any $\varepsilon>0$ and $x\in D(T)$ with
$\|Tx\|\leq\varepsilon$, there exists $x_0\in D(T)$ such that
$Tx_0=0$ and $\|x-x_0\|\leq K\varepsilon$.\\If we denote the null
space of $T$ by $N(T)$, then the condition (b) is equivalent
to:\\
(c) For any $x\in D(T)$, there exists $x_0\in N(T)$ such that
$\|x-x_0\|\leq K\|Tx\|$.
\end{remark}

Let $X, Y$ be Banach spaces. Let $L(X, Y)$, $C(X, Y)$ and $B(X, Y)$
denote the linear space of all linear operators, the homogeneous set
of all closed linear operators with a dense domain and the Banach
space of all bounded linear operators from $X$ into $Y$,
respectively. The identity operator is denoted by $I$. Let us
introduce the reduced minimum modulus of closed linear operators.
\begin{definition}\label{Def1.2} \cite{11} The reduced minimum modulus of $T\in C(X, Y)$ is defined by
$$\gamma(T)=\inf\{\|Tx\|: \quad x\in D(T)\ with \ d(x,N(T)):=\inf_{z\in N(T)}d(x,z)=1\}.$$
\end{definition}
\ It is easy to see that $\gamma(T)=\sup\{\alpha\geq 0: \|Tx\|\geq\alpha\ d(x,N(T)),\ x\in D(T)\}$ and if $X$ and $Y$ are Hilbert spaces, then
$$\gamma(T)=\inf\{\|Tx\|: \quad x\in D(T)\cap {N(T)}^\perp\ with \ \|x\|=1\},$$
where $\perp$ denotes the orthogonal complement in Hilbert spaces;
see also \cite{2}.
\begin{theorem}\label{The1.1} \cite{5} Let $X, Y$ be Hilbert spaces and $T\in C(X,Y)$. Then
$T$ has the Hyers--Ulam stability if and only if it has closed range.
In this case, $K_T=\gamma(T)^{-1}$.
\end{theorem}

Let us introduce the notion of a generalized inverse (see e.g. \cite{g1}) and that of the Moore--Penrose inverse of a closed operator.

\begin{definition}\label{Def1.3} \cite{g1} An operator $T\in C(X, Y)$ possesses a (bounded) generalized inverse if there exists an operator $S \in B(Y, X)$ such that $R(S)\subseteq D(T)$ and (1) $TSTx = Tx$ for all $x \in D(T)$; (2) $STSy = Sy$ for all $y \in Y$; (3) $ST$ is continuous.
We denote a generalized inverse of $T$ by $T^+$.
\end{definition}
 In general, the generalized inverse need not exist and is not unique even if it exists. We need the following lemma concerning the existence of generalized inverses.
\begin{lemma}\label{Lem1.1} \cite{17}
$(a)$ Let $T\in C(X,Y)$.
Suppose that $N(T)$ has a topological complement $N(T)^{c}$ in $X$
and $\overline{R(T)}$ has a topological complement
$\overline{R(T)}^{c}$ in $Y$, i.e., $$X=N(T)\oplus N(T)^{c}\quad and
\quad Y=\overline{R(T)}\oplus \overline{R(T)}^{c}.$$ Let $P$ denote
the projector of $X$ onto $N(T)$ along $N(T)^{c}$ and $Q$ denote the
projector of $Y$ onto $\overline{R(T)}$ along $\overline{R(T)}^{c}$.
Then there is a unique $S \in C(Y,X)$ satisfying: $1)$
$TST=T$ on $D(T)$;\quad $2)$ $STS=S$ on $D(S)$; $3)$ $ST=I-P$ on
$D(T)$ and $4)$ $TS=Q$ on $D(S)$, where $D(S)=R(T)+\overline{R(T)}^{c}$.
\par\noindent $(b)$ Under the assumptions of part
$(a)$, $S$ is bounded if and only if\ $R(T)$ is closed. In this case, $S$ is a bounded generalized inverse of $T$ with $D(S)=Y$, $N(S)=R(T)^c$ and $R(S)=D(T)\cap N(T)^c$.
\end{lemma}

\begin{definition}\label{Def1.4}\cite{17} Let $X$,$Y$ be Hilbert spaces and $T\in C(X,
Y)$. If the topological decompositions in Lemma \ref{Lem1.1} are orthogonal,
i.e., $$X=N(T)\dot{+} N(T)^{\perp}\quad and \quad Y=\overline{R(T)} \dot{+} {R(T)}^{\perp},$$ where $\dot{+}$ denotes the orthogonal direct sum, then the
corresponding generalized inverse is usually called the
Moore--Penrose inverse of $T$. In this case, the operators $P$ and
$Q$ in Lemma \ref{Lem1.1} are orthogonal projectors. The Moore--Penrose inverse of $T$ is always denoted by
$T^\dagger$.
\end{definition}
\begin{remark} The operator $T\in C(X,Y)$
has a generalized inverse  $T^+\in B(Y,X)$ if and only if $$X=N(T)\oplus N(T)^{c}\quad and
\quad Y=R(T)\oplus R(T)^{c}.$$
In this case,
it follows from the closed graph theorem that the operator $TT^+$ is
a projector from $Y$ onto $R(T)$ such that $N(TT^+)=N(T^+)$ and
$R(TT^+)=R(T).$ Meanwhile, by the condition $(3)$ in Definition \ref{Def1.3},
 $T^+T$ can be extended uniquely to a projector from $X$ onto
$\overline{R(T^+)}$ with the null space $N(T)$ and the range
$\overline{R(T^+)}.$
\end{remark}
\par  It is well known that the perturbation analysis of
Moore--Penrose inverses and generalized inverses in Hilbert and
Banach spaces are very important in practical applications of
operator theory and has been widely studied; cf. \cite{1, 2, 4, 6,
7, 8, 17, 21, 22}. Recently, the perturbation of generalized
inverses for linear operators in Hilbert spaces or Banach spaces has
been studied in \cite{2, 4, 6, 7, 8, 21, 22}. To achieve our
results, we need the concept of $T-$boundedness as follows.
\begin{definition}\label{Def1.5} \cite{11} Let $T$ and $P$ be linear
operators with the same domain space such that $D(T)\subseteq D(P)$
and $$\|Px\|\leq a\|x\|+b\|Tx\| \quad\quad (x\in D(T)),$$
where $a,b$ are nonnegative constants. Then we say $P$ is relatively
bounded with respect to $T$ or simply $T$-bounded and the greatest lower
bound of all possible constants $b$ will be called the relative
bound of $P$ with respect to $T$ or simply the $T$-bound.
\end{definition}
\begin{theorem}\label{The1.2} \cite{8} Let $X$, $Y$ be Banach spaces and let $T\in C(X,Y)$ with a bounded
 generalized inverse $T^+\in B(Y,X)$. Let $\delta T\in
L(X,Y)$ be $T-$bounded with constants $a,b$. If
 $a\|T^+\|+b\|TT^+\|<1$,
 then the following statements are
equivalent: \par\noindent $(1)$\ $B:=T^+(I+\delta TT^+)^{-1}:\ Y\to X$
 is a bounded generalized inverse of $\overline{T}:=T+\delta T;$
 \par\noindent $(2)$\ $(I+\delta T T^+)^{-1}R(\overline{T})= R(T);$
 \par\noindent
$(3)$\ $(I+\delta TT^+)^{-1}\overline{T}$ maps $N(T)$ into $R(T);$
\par\noindent $(4)$  $R(\overline{T})\cap N(T^+)=\{0\}.$\par
 Moreover, if one of the conditions above is true, then $\overline{T}$ is a closed operator and its range $R(\overline{T})$ is closed.
\end{theorem}
 \par In this paper, we use the expression and the stability characterization of the Moore--Penrose inverse to investigate the condition for
the perturbed operators to have the Hyers--Ulam stability and the
condition in order that the Hyers--Ulam stability constant be
continuous. In Section 2, we first establish a relationship between
the Hyers--Ulam stability and the Moore--Penrose inverse, that is, a
closed operator has the Hyers--Ulam stability if and only if it has
the bounded Moore--Penrose inverse. Meanwhile, the stability
constant is determined in terms of the Moore--Penrose inverse.
Utilizing this result, we give some sufficient conditions for the
perturbed operators having the Hyers--Ulam stability and give an
explicit expression of the Hyers--Ulam stability constant. In the
case of bounded linear operators, some sufficient and necessary
conditions for the Hyers--Ulam stability constants to be continuous
are also provided in Section 3. In the end, as an application, we
give a characterization for the Hyers--Ulam stability constants of
the semi-Fredholm operators to be continuous.

\section{The case of closed linear operators}

The following lemma is proved by a straightforward verification of the conditions of the definition of Moore--Penrose inverse.
\begin{lemma}\label{Lem2.1} \cite{8}
Let $X, Y$ be Hilbert spaces and $T\in C(X,Y)$ with a generalized inverse $T^+\in B(Y,X)$, then
$T$ has the bounded Moore--Penrose inverse $T^\dagger$ and
$$T^\dagger=[I-P_{N(T)}^\perp]T^+P_{R(T)}^\perp.$$
\end{lemma}

 \begin{theorem}\label{The2.1} Let $X, Y$ be Hilbert spaces and $T\in C(X,Y)$. Then the following statements are equivalent:
\par\noindent\ $(1)$\ $T$ has the Hyers--Ulam stability;
 \par\noindent\ $(2)$\ $T$ has the bounded Moore--Penrose inverse $T^\dagger$;
 \par\noindent\ $(3)$\ $T$ has a bounded generalized inverse $T^+$;
 \par\noindent\ $(4)$\ $T$ has closed range.
\par Moreover, if one of the conditions above is true, then $D(T^\dagger)=Y$, $N(T^\dagger)=R(T)^{\perp}$,
$R(T^\dagger)=D(T)\cap N(T)^{\perp}$ and
$$K_T=\|T^\dagger\|=\gamma(T)^{-1}.$$
\end{theorem}
\begin{proof} Note that $T|_{{N(T)}^\perp}: D(T)\cap{N(T)}^\perp\to R(T)$ is invertible and $T^\dagger$ is defined by $$T^\dagger y=(T|_{{N(T)}^\perp})^{-1}Q y\quad\quad (y\in R(T)+{R(T)}^\bot),$$
where $Q$ is the orthogonal
projector of $Y$ onto $\overline{R(T)}$ along ${R(T)}^\bot$. Then $T^\dagger$ is a densely defined closed operator with $D(T^\dagger)=R(T)+{R(T)}^\bot$ and $R(T^\dagger)=D(T)\cap N(T)^{\perp}$. \par $(2)\Rightarrow (1).$ If $T$ has the bounded Moore--Penrose inverse $T^\dagger$, then for all $x\in D(T)$, $(I-T^\dagger T)x\in N(T)$ and
$$\|x-(I-T^\dagger T)x\|=\|T^\dagger Tx\|\leq\|T^\dagger\| \|Tx\|.$$
This means that $T$ has the Hyers--Ulam stability and
$K_T\leq\|T^\dagger\|.$ Assume that $K$ is a Hyers--Ulam stability
constant for $T$, i.e., for all $x\in D(T)$, there exists $x_0\in
N(T)$ such that $\|x-x_0\|\leq K\|Tx\|$. Then for all $y\in Y$,
$T^\dagger y\in D(T)\cap{N(T)}^\perp$ and there exists $x_1\in N(T)$
such that $\|T^\dag y-x_1\|\leq K\|TT^\dag y\|$. Since $T^\dagger y
\bot x_1$, we get
\begin{eqnarray}\label{(2.1)}
\|T^\dagger y\|\leq \|T^\dag y-x_1\|\leq K\|TT^\dag y\|\leq K\|y\|.
\end{eqnarray}
Hence $K\geq\|T^\dag\|$ and thus $K_T\geq\|T^\dag\|$. Therefore
$K_T=\|T^\dag\|$. \par $(1)\Rightarrow (2).$ If $T$ has the
Hyers--Ulam stability, then by (\ref{(2.1)}), we know that the
Moore--Penrose inverse $T^\dagger$ is bounded.
\par $(2)\Leftrightarrow (3).$ From Lemma \ref{Lem2.1}, we can see $(3)\Rightarrow (2)$.  Since the Moore--Penrose inverse $T^\dagger$ is also a generalized inverse, we can get $(2)\Rightarrow (3)$.
\par $(4)\Rightarrow (2).$
If $R(T)$ is closed, then $D(T^\dagger)=Y$ and by the Closed Graph Theorem, $T^\dagger$ is bounded. \par $(2)\Rightarrow (4).$
If $T^\dagger$ is bounded, since $T^\dagger$ is a densely defined closed operator, we get $D(T^\dagger)=Y$, i.e, $Y=R(T)+{R(T)}^\bot.$ This implies that $R(T)$ is closed. In the following, we shall show $\gamma(T)=\|T^\dagger\|^{-1}$. In fact, for all $x\in D(T)$, we have
$(I-T^\dagger T)x\in N(T)$ and
$$d(x, N(T))\leq\|x-(I-T^\dagger T)x\|=\|T^\dagger Tx\|\leq\|T^\dagger\| \|Tx\|.$$ Then
$\gamma(T)\geq\|T^\dagger\|^{-1}$. Since $\gamma(T)\leq\|Tx\|$ for all $x\in D(T)\cap{N(T)}^\perp$ with $\|x\|=1$, we get
for all $y\in Y$ satisfying $\|T^\dagger y\|=1$, $$\gamma(T)\leq\|TT^\dagger y\|=\|Qy\|\leq\|y\|.$$ Hence for all $y\in Y$ with $T^\dagger y\neq 0$, $\gamma(T)\leq\frac{\|y\|}{\|T^\dagger y\|}$. Thus
$$\gamma(T)\leq\inf\left\{\frac{\|y\|}{\|T^\dagger y\|}:\ y\in Y, T^\dagger y\neq 0\right\}=\left(\sup\left\{\frac{\|T^\dagger y\|}{\|y\|}:\ y\in Y\right\}\right)^{-1}=\|T^\dagger\|^{-1}$$
Therefore $\gamma(T)=\|T^\dagger\|^{-1}$.
\end{proof}
\par\vspace{0.1in} Applying Theorem \ref{The2.1}, we get \cite[Theorem 3.1]{5}.
\begin{corollary}\label{Cor2.1} \cite{5} Let $X, Y$ be Hilbert spaces and $T\in C(X,Y)$. Then
$T$ has the Hyers--Ulam stability if and only if $T$ has closed range.
 Moreover, in this case, $K_T=\gamma(T)^{-1}$.
\end{corollary}

In the following, we shall use the expressions and stability
characterizations of the Moore--Penrose inverse to investigate the
Hyers--Ulam stability of closed operators. We need the following
lemma, which can be proved by using the fact that
$P_MP_M^*[I-(P_M-P_M^*)^2]=[I-(P_M-P_M^*)^2]P_MP_M^*$.

\begin{lemma}\label{Lem2.2} \cite{8}
Let $X$ be a Hilbert space and $M$ be a closed linear subspace of
$X$. Let $P_M: X\to M$ be a (not necessarily selfadjoint) projector from $X$ onto $M$,
then the orthogonal projector $P_M^\perp$ from $X$ onto $M$ can be
expressed by
$$P_M^\perp=P_MP_M^*[I-(P_M-P_M^*)^2]^{-1}=[I-(P_M-P_M^*)^2]^{-1}P_MP_M^*.$$
\end{lemma}

Utilizing Lemma \ref{Lem2.1} with Lemma \ref{Lem2.2}, we can get the following lemma.
\begin{lemma}\label{Lem2.3}
Let $T\in C(X,Y)$ with a bounded generalized inverse $T^+\in B(Y,X)$, then
$T$ has the bounded Moore--Penrose inverse $T^\dagger$ and
$$T^\dagger=\{I-[(T^+T)^{**}-(T^+T)^*]^2\}^{-1}(T^+T)^*T^+(TT^+)^*\{I-[TT^+-(TT^+)^*]^2\}^{-1}$$
\end{lemma}
\begin{proof} Since $TT^+$ is a projector from $Y$ onto $R(T)$, it follows from Lemma \ref{Lem2.2} that
$$P_{R(T)}^\perp=TT^+(TT^+)^*\{I-[TT^+-(TT^+)^*]^2\}^{-1}.$$
Noting that $I-T^+T$ is a bounded projector from $D(T)$ onto $N(T)$ and $D(T)$ is dense in $X$,
one can verify that $(I-T^+T)^*$ is defined on the whole space $Y$. It follows from the Closed Graph Theorem that $(I-T^+T)^*$ is bounded.
Hence $P_{N(T)}=(I-T^+T)^{**}$, which is exactly the unique norm-preserving extension
 to whole space $X$ of $I-T^+T$. Thus  $ P_{N(T)}^*=(I-T^+T)^*$ and
\begin{eqnarray*}&&I-P_{N(T)}^\perp\\ &=&I-[I-(P_{N(T)}-P_{N(T)}^*)^2]^{-1}P_{N(T)}P_{N(T)}^*\\
&=&[I-(P_{N(T)}-P_{N(T)}^*)^2]^{-1}[I-(P_{N(T)}-P_{N(T)}^*)^2-P_{N(T)}P_{N(T)}^*]\\
&=&[I-(P_{N(T)}-P_{N(T)}^*)^2]^{-1}(I-P_{N(T)}^*)(I-P_{N(T)})\\
&=&\{I-[(I-T^+T)^{**}-(I-T^+T)^*]^2\}^{-1}[I-(I-T^+T)^*][I-P_{N(T)}]\\
&=&\{I-[(T^+T)^{**}-(T^+T)^*]^2\}^{-1}(T^+T)^*[I-P_{N(T)}].
\end{eqnarray*}
Therefore, by $P_{N(T)}|_{D(T)}=I-T^+T$ and $R(T^+)\subset D(T)$, we obtain
\begin{eqnarray*}T^\dagger&=&[I-P_{N(T)}^\perp]T^+P_{R(T)}^\perp\\
&=&\{I-[(T^+T)^{**}-(T^+T)^*]^2\}^{-1}(T^+T)^*[I-P_{N(T)}]\\&&T^+(TT^+)(TT^+)^*\{I-[TT^+-(TT^+)^*]^2\}^{-1}
\\
&=&\{I-[(T^+T)^{**}-(T^+T)^*]^2\}^{-1}(T^+T)^*T^+(TT^+)^*\{I-[TT^+-(TT^+)^*]^2\}^{-1}.
\end{eqnarray*}
\end{proof}

\begin{lemma}\label{Lem2.4}
Let $T\in C(X,Y)$ have a
generalized inverse $T^+\in B(Y,X)$ and $\delta T\in
L(X,Y)$ be $T-$bounded with constants $a,b$. If
 $a\|T^+\|+b\|TT^+\|<1$,
 then $\overline{T}=T+\delta T$ is closed and $$B=T^+(I+\delta
TT^+)^{-1}:\ Y\to X$$ satisfies
$B\overline{T}B=B$ on $Y$, $R(B)=R(T^+)$ and
 $N(B)=N(T^+)$.
\end{lemma}
\begin{proof} It follows from Theorem \ref{The1.2} that $\overline{T}$ is closed. Since
\begin{eqnarray*}
\|\delta TT^+ y\|&\leq&a\|T^+ y\|+b\|TT^+ y\|\\
&\leq&(a\|T^+\|+b\|TT^+\|)\|y\|\,,
\end{eqnarray*}
for all $y \in Y$, we get
$
\|\delta TT^+\|\leq
 a\|T^+\|+b\|TT^+\|<1.
$
By the celebrated Banach Lemma, the inverse of
$I+\delta TT^+$ exists and $(I+\delta TT^+)^{-1}\in B(Y)$.
 Hence $B=T^+(I+\delta
TT^+)^{-1}: Y\to X$ is a bounded linear operator. It is easy to verify $R(B)=R(T^+)$ and
 $N(B)=N(T^+)$.
 To the end, we need to
show $B\overline{T}B=B$ on $Y$. Indeed, $R(B)=R(T^+)\subseteq D(T)=D(\overline{T})$. Since
\begin{eqnarray*}
(I+\delta TT^+)TT^+=TT^++\delta TT^+=(T+\delta T)T^+=\overline{T}T^+,
\end{eqnarray*}
we have $TT^+=(I+\delta TT^+)^{-1}\overline{T}T^+$ and
therefore,\begin{eqnarray*}
B\overline{T}B&=&T^+(I+\delta TT^+)^{-1}\overline{T}T^+(I+\delta
TT^+)^{-1} \\
 &=&T^+TT^+(I+\delta
TT^+)^{-1}=T^+(I+\delta TT^+)^{-1}=B.\end{eqnarray*}
\end{proof}

In the next theorem some conditions are given for $\overline{T}$ to have the Hyers--Ulam stability as well as the Hyers--Ulam stability constant $K_{\overline{T}}$ is explicitly expressed.

\begin{theorem}\label{The2.2} Let $X$, $Y$ be Hilbert spaces and let $T\in C(X,Y)$ have the Hyers--Ulam stability. Let $T^+\in B(Y,X)$  be a bounded generalized inverse of $T$ and  $\delta T\in
L(X,Y)$ be $T-$bounded with constants $a,b$. If
 $a\|T^+\|+b\|TT^+\|<1$,
 then the following statements are
equivalent: \par\noindent $(1)$\ $B=T^+(I+\delta TT^+)^{-1}:\ Y\to X$
 is a bounded generalized inverse of $\overline{T}=T+\delta T;$
 \par\noindent $(2)$\ $(I+\delta T T^+)^{-1}R(\overline{T})= R(T);$
 \par\noindent
$(3)$\ $(I+\delta TT^+)^{-1}\overline{T}$ maps $N(T)$ into $R(T);$
\par\noindent $(4)$ $R(\overline{T})\cap N(T^+)=\{0\}.$ \par
 Moreover, if one of the conditions above is true, then $R(\overline{T})$ is closed, $\overline{T}$ has the Hyers--Ulam stability and $K_{\overline{T}}=\|\overline{T}^\dagger\|$, where
 \begin{eqnarray*}
\overline{T}^\dagger&=&\{I-[(T^+(I+\delta T T^+)^{-1}\overline{T})^{**}-(T^+(I+\delta T T^+)^{-1}\overline{T})^*]^2\}^{-1}\\
&&
[T^+(I+\delta T T^+)^{-1}\overline{T}]^*T^+(I+\delta T T^+)^{-1}[\overline{T}T^+(I+\delta T T^+)^{-1}]^* \\
&&\{I-[\overline{T}T^+(I+\delta T T^+)^{-1}-(\overline{T}T^+(I+\delta T T^+)^{-1})^*]^2\}^{-1}.
 \end{eqnarray*}
\end{theorem}
\begin{proof}  From Theorem \ref{The1.2}, we can see the equivalence. If one of the conditions is true, then $R(\overline{T})$ is closed and $T^+(I+\delta TT^+)^{-1}:\ Y\to X$
 is a bounded generalized inverse of $T+\delta T$. By Lemma \ref{Lem2.3}, we can get what we desired.\end{proof}
 \begin{remark}\label{Rem2.1} It should be noted that under any one of the conditions in Theorem \ref{The2.2}, the Hyers--Ulam stability constant $K_{\overline{T}}$ is continuous at $T$ if $\delta T$
 is bounded and $\|\delta T\|\to 0$. In fact,
 \begin{eqnarray*}&&T^+(I+\delta T T^+)^{-1}\overline{T}=(I+T^+\delta T)^{-1}T^+\overline{T}\\
 &=&(I+T^+\delta T)^{-1}T^+T+(I+T^+\delta T)^{-1}T^+\delta T\rightarrow T^+T\end{eqnarray*}
 and\begin{eqnarray*} &&
 \overline{T}T^+(I+\delta T T^+)^{-1}\\&=&TT^+(I+\delta T T^+)^{-1}+\delta TT^+(I+\delta T T^+)^{-1}\rightarrow TT^+\,,
\end{eqnarray*}
as $\|\delta T\|\to 0$.
\end{remark}
 \begin{corollary}\label{Cor2.2} Let $X$, $Y$ be Hilbert spaces and let $T\in C(X,Y)$ have the Hyers--Ulam stability. Let $T^+\in B(Y,X)$  be a bounded generalized inverse of $T$ and  $\delta T\in
L(X,Y)$ be $T-$bounded with constants $a,b$. If
 $a\|T^+\|+b\|TT^+\|<1$ and $N(T)\subseteq N(\delta T),$ then $\overline{T}=T+\delta T$ has the Hyers--Ulam stability and $K_{\overline{T}}=\|\overline{T}^\dagger\|,$ where
\begin{eqnarray*}
\overline{T}^\dagger&=&\{I-[(T^+T)^{**}-(T^+T)^*]^2\}^{-1}(T^+T)^*T^+(I+\delta T T^+)^{-1}[\overline{T}T^+(I+\delta T T^+)^{-1}]^*
\\&&\{I-[\overline{T}T^+(I+\delta T T^+)^{-1}-((\overline{T}T^+(I+\delta T T^+)^{-1})^*]^2\}^{-1}.
 \end{eqnarray*}
\end{corollary}
\begin{proof} We first show that $N(\overline{T})=N(T)$. In fact, by $
N(T)\subseteq N(\delta T)$, we have $ N(T)\subseteq
N(\overline{T})$ as well as $\delta T(I-T^+T)=0$, whence
$\delta T=\delta TT^+T$. If $x\in N(\overline{T})$, then
$$0=\overline{T}x=Tx+\delta Tx=Tx+\delta TT^+Tx=
 (I+\delta TT^+)Tx.$$ Since $I+\delta TT^+$ is invertible,
 $Tx=0$. Hence $N(\overline{T})=N(T)$. Thus
 $P_{N(\overline{T})}^\perp=P_{N(T)}^\perp$ and
 $P_{N(\overline{T})}^\perp|_{D(T)}=I-\{I-[(I-T^+T)^{**}-(I-T^+T)^*]^2\}^{-1}[I-(I-T^+T)^*]T^+T$. Therefore, by Lemma \ref{Lem2.3},
we can get what we desired.
\end{proof}
 \begin{corollary}\label{Cor2.3}
Let $X$, $Y$ be Hilbert spaces and let $T\in C(X,Y)$ have the Hyers--Ulam stability. Let $T^+\in B(Y,X)$  be a bounded generalized inverse of $T$ and  $\delta T\in
L(X,Y)$ be $T-$bounded with constants $a,b$. If
 $a\|T^+\|+b\|TT^+\|<1$ and  $R(\delta T)\subseteq R(T)$,
 then $\overline{T}=T+\delta T$ has the Hyers--Ulam stability and $K_{\overline{T}}=\|\overline{T}^\dagger\|,$ where\begin{eqnarray*}
\overline{T}^\dagger&=&\{I-[(T^+(I+\delta T T^+)^{-1}\overline{T})^{**}-(T^+(I+\delta T T^+)^{-1}\overline{T})^*]^2\}^{-1}\\
&&
[T^+(I+\delta T T^+)^{-1}\overline{T}]^*T^+(I+\delta T T^+)^{-1}(TT^+)^*\{I-[TT^+-(TT^+)^*]^2\}^{-1}.
 \end{eqnarray*}
\end{corollary}
\begin{proof} We first show $R(\overline{T})=R(T)$. In fact, by $R(\delta
T)\subseteq R(T)$, we have $R(\overline{T})\subseteq R(T)$ and also $(I-
TT^+)\delta T=0$, which implies $\delta T=TT^+\delta T
$. Then
$$
\overline{T}T^+=(T+\delta
T)T^+=TT^++TT^+\delta
TT^+=TT^+(I+\delta TT^+)
$$
and $\overline{T}T^+(I+\delta TT^+)^{-1}=TT^+$.
Hence $T=TT^+T=\overline{T}T^+(I+\delta
TT^+)^{-1}T.$ This means $R(T)\subseteq R(\overline{T})$.
Thus $R(\overline{T})=R(T)$ and
$$P_{R(\overline{T})}^\perp=P_{R(T)}^\perp=TT^+(TT^+)^*\{I-[TT^+-(TT^+)^*]^2\}^{-1}.$$  By $R(I-TT^+)=N(T^+)=N(B)$, we have
$T^+(I+\delta TT^+)^{-1}TT^+=T^+(I+\delta TT^+)^{-1}$. Therefore by Lemma \ref{Lem2.3},
we can get what we desired.
\end{proof}
\begin{remark}\label{Rem2.2} It should be noted that if $N(T)\subseteq N(\delta T) $ or $R(\delta T)\subseteq R(T)$ holds, then the Hyers--Ulam stability constant $K_{\overline{T}}$ is continuous at $T$ as $\delta T$
 is bounded and $\|\delta T\|\to 0$.
\end{remark}

\noindent Recall that a closed operator $T: X \to Y$ is called left
semi-Fredholm if $\dim N(T)<\infty$ and $R(T)$ is closed. It is
called right semi-Fredholm if ${\rm codim}R(T)<\infty$ and $R(T)$ is
closed. We say a closed operator $T$ is semi-Fredholm if it is left
or right semi-Fredholm. The following result involving semi-Fredholm
operators holds.

\begin{theorem}\label{The2.3} Let $X$, $Y$ be Hilbert spaces and let $T\in C(X,Y)$ be a semi-Fredholm operator. Let $T^+\in B(Y,X)$  be a bounded generalized inverse of $T$ and  $\delta T\in
L(X,Y)$ be $T-$bounded with constants $a,b$. If
 $a\|T^+\|+b\|TT^+\|<1$ and
 $$ either \ \dim N(\overline{T})=\dim N(T)<\infty\quad or \quad {\rm codim}R(\overline{T})={\rm codim}R(T)<\infty,$$
 then $\overline{T}=T+\delta T$ has the Hyers--Ulam stability and $K_{\overline{T}}=\|\overline{T}^\dagger\|$, where
 \begin{eqnarray*}
\overline{T}^\dagger&=&\{I-[(T^+(I+\delta T T^+)^{-1}\overline{T})^{**}-(T^+(I+\delta T T^+)^{-1}(T+\delta T))^*]^2\}^{-1}\\
&&
[T^+(I+\delta T T^+)^{-1}\overline{T}]^*T^+(I+\delta T T^+)^{-1}[\overline{T}T^+(I+\delta T T^+)^{-1}]^*\\
&& \{I-[\overline{T}T^+(I+\delta T T^+)^{-1}-(\overline{T}T^+(I+\delta T T^+)^{-1})^*]^2\}^{-1}.
 \end{eqnarray*}
\end{theorem}
\begin{proof} By Lemma \ref{Lem2.4}, $B=B\overline{T}B$ on $Y$, where $B=T^+(I+\delta
TT^+)^{-1}$. Then $B\overline{T}B\overline{T}=B\overline{T}$ and it
follows that $B\overline{T}$ is an idempotent on $D(\overline{T})$.
Due to $\overline{T}(I-T^+ T)=(T+\delta T)(I-T^+ T)=\delta T(I-T^+
T)$ and
\begin{eqnarray*}\|[\delta T (I-T^+ T)]x\|&\leq&
a\|(I-T^+ T)x\|+b\|T(I-T^+T)x\|\\ &\leq&
a\|I-T^+ T\|\cdot\|x\|\qquad(x \in D(T))
\end{eqnarray*}
we conclude that
 \begin{eqnarray*}
B\overline{T}&=&T^+[I+(\overline{T}-T)T^+]^{-1}\overline{T}\\
&=&T^+[I+(\overline{T}-T)T^+]^{-1}\overline{T}T^+T+T^+[I+(\overline{T}-T)T^+]^{-1}\overline{T}(I-T^+T)\\
&=&T^+T+T^+[I+(\overline{T}-T)T^+]^{-1}\overline{T}(I-T^+T)
 \end{eqnarray*}
is bounded. Hence $B\overline{T}$ can be extended uniquely from $D(\overline{T})$ to $X$. We denote its extension by $S$, which is defined by $Sx=\lim\limits_{n\rightarrow +\infty}B\overline{T}x_n$ for all $x\in X$, where $x_n\in D(\overline{T})$ satisfies $x_n\rightarrow x$.  It is easy to verify $S\in B(X)$ and  $S^2=S$. Since $R(B\overline{T})\subset R(S)\subset\overline{R(B\overline{T})}$, we get $R(S)=\overline{R(B\overline{T})}=\overline{R(B)}=\overline{R(T^+)}$. Next, for any $x\in N(S)$, then there exists $\{x_n\}\subset D(\overline{T})$ with $x_n\rightarrow x$ such that $B\overline{T}x_n\rightarrow Sx=0$. Hence
 $x_n-B\overline{T}x_n\in  N(B\overline{T})$ satisfies $x_n-B\overline{T}x_n\rightarrow x$, which implies  $N(S)\subset\overline{N(B\overline{T})}$. Because $N(B\overline{T})\subset N(S)$ and $N(S)$ is closed,  $\overline{N(B\overline{T})}\subset N(S)$. Thus we conclude  $N(S)=\overline{N(B\overline{T})}$.
On the other hand, by the Closed Graph Theorem, we can observe that $\overline{T}B$ is a projector from $Y$ onto $R(\overline{T}B)$
 and $N(\overline{T}B)=N(B)=N(T^+)$. Thus by $X=R(S)\oplus N(S)$ and $Y=R(\overline{T}B)\oplus N(B)$, we get
\begin{eqnarray}\label{(2.2)}
\overline{R(T^+)}\oplus N(T)=X=\overline{R(B\overline{T})}\oplus \overline{N(B\overline{T})}=\overline{R(T^+)}\oplus \overline{N(B\overline{T})}
\end{eqnarray}
and
\begin{eqnarray}\label{(2.3)}
R(T)\oplus N(T^+)=Y=R(\overline{T}B)\oplus N(T^+)=R(\overline{T})+N(T^+)=R(\overline{T})\oplus N^-,
\end{eqnarray}
where $N^-$ satisfies $N(T^+)=N^-\oplus(R(\overline{T})\cap
N(T^+))$. If $\dim N(\overline{T})=\dim N(T)<\infty$, then by
(\ref{(2.2)}), we get $\dim \overline{N(B\overline{T})}=\dim
N(\overline{T})$. Noting $N(\overline{T})\subseteq
\overline{N(B\overline{T})}$, we can obtain
$N(\overline{T})=\overline{N(B\overline{T})}$. Hence for any $y\in
R(\overline{T})\cap N(T^+)$, there exists $x\in D(\overline{T})$
such that $y=\overline{T}x$ and $T^+\overline{T}x=0$. This means
$\overline{T}x\in N(T^+)=N(B)$. Thus $x\in N(B\overline{T})$ and so
$x\in N(\overline{T})$, which implies $y=\overline{T}x=0$. Therefore
$R(\overline{T})\cap N(T^+)=\{0\}$. If ${\rm
codim}R(\overline{T})={\rm codim}R(T)<\infty,$ then by
(\ref{(2.3)}), $\dim N(T^+)=\dim N^-$ and so $R(\overline{T})\cap
N(T^+)=\{0\}$. Using Theorem \ref{The2.2} we can get what we
desired.\end{proof}

\section{The case of bounded linear operators}

In this section, we shall give some sufficient and necessary conditions for the Hyers--Ulam stability constants to be continuous.
\begin{theorem}\label{The3.1} Let $X$, $Y$ be Hilbert spaces and let $T\in B(X,Y)$ have the Hyers--Ulam stability, i.e, $T$ has a
Moore--Penrose inverse $T^\dagger\in B(Y,X)$. If $\delta T\in
B(X,Y)$  satisfies
$\|\delta T\| \|T^\dagger\| <1$, then the following statements are
equivalent: \par\noindent $(1)$\ $B=T^\dagger(I+\delta TT^\dagger)^{-1}:\ Y\to X$
 is a bounded generalized inverse of $\overline{T}=T+\delta T;$
 \par\noindent $(2)$\ $(I+\delta T T^\dagger)^{-1}R(\overline{T})= R(T);$
 \par\noindent
$(3)$\ $(I+\delta TT^\dagger)^{-1}\overline{T}$ maps $N(T)$ into $R(T);$
\par\noindent
$(4)$ $(I+T^\dagger\delta T)^{-1}N(T)=N(\overline{T});$
\par\noindent $(5)$ $R(\overline{T})\cap N(T^\dagger)=\{0\};$ \par\noindent
$(6)$\ $\overline{T}$ has the Hyers--Ulam stability and the Hyers--Ulam stability constant $K_{\overline{T}}$
satisfies $$\lim_{\|\delta T\|\to 0}K_{\overline{T}}=K_T;$$
$(7)$\ $\overline{T}$ has the Hyers--Ulam stability and there exist $M>0$ and $\varepsilon>0$ such that
$K_{\overline{T}}\leq M$ for all $\|\delta T\|<\varepsilon$; \par\noindent
 $(8)$\ $\overline{T}$ has the bounded Moore--Penrose inverse $\overline{T}^\dagger\in B(Y,X)$ with $\lim\limits_{\|\delta T\|\to 0}\overline{T}^\dagger=T^\dagger.$ \par
 In this case, $K_{\overline{T}}=\|\overline{T}^\dagger\|$ and
 \begin{eqnarray*}
\overline{T}^\dagger&=&\{I-[T^+(I+\delta T T^+)^{-1}\overline{T}-(T^+(I+\delta T T^+)^{-1}\overline{T})^*]^2\}^{-1}\\
&&
[T^+(I+\delta T T^+)^{-1}\overline{T}]^*T^+(I+\delta T T^+)^{-1}[\overline{T}T^+(I+\delta T T^+)^{-1}]^*\\
&& \{I-[\overline{T}T^+(I+\delta T T^+)^{-1}-(\overline{T}T^+(I+\delta T T^+)^{-1})^*]^2\}^{-1}.
 \end{eqnarray*}
\end{theorem}
\begin{proof} In Theorem \ref{The2.2}, we take $a=\|\delta T\|$ and $b=0$. Then it follows  that (1), (2), (3) and (5) are equivalent. The equivalence of (4) and (5) can be found in \cite{7}.
 By Theorem \ref{The2.1} and Theorem \ref{The2.2}, we can see $(5) \Rightarrow (8)$ and $(8) \Rightarrow (6)$. Obviously, $(6) \Rightarrow (7)$.
 Next we shall show $(7)\Rightarrow (8)$ and $(8)\Rightarrow (5)$.
 Since
\begin{eqnarray*}
 &&\overline{T}^\dag-T^\dag\\
 & =&\overline{T}^\dag(I-TT^\dag)+\overline{T}^\dag(T-\overline{T})T^\dag
 +(\overline{T}^\dag \overline{T}-I)T^\dag\\
 & =&\overline{T}^\dag \,\overline{T}\,\overline{T}^\dag(I-TT^\dag)+\overline{T}^\dag(T-\overline{T})T^\dag
+(\overline{T}^\dag\,\overline{T}-I)T^\dag TT^\dag\\
 & =&\overline{T}^\dag[(I-TT^\dag)\overline{T}\,\overline{T}^\dag]^*+\overline{T}^\dag(T-\overline{T})T^\dag
+[T^\dag T(\overline{T}^\dag \overline{T}-I)]^*T^\dag\\
 &=&\overline{T}^\dag[(I-TT^\dag)(\overline{T}-T)\overline{T}^\dag]^*+\overline{T}^\dag(T-\overline{T})T^\dag+[T^\dag (T-\overline{T})(\overline{T}^\dag \overline{T}-I)]^*T^\dag\\
 & =&\overline{T}^\dag(\overline{T}^\dag)^*(\overline{T}-T)^*(I-TT^\dag)+\overline{T}^\dag(T-\overline{T})T^\dag
+(\overline{T}^\dag
 \overline{T}-I)(T-\overline{T})^*(T^\dag)^*T^\dag,
\end{eqnarray*}
we can get
$$||\overline{T}^\dag-T^\dag||\leq
 (||\overline{T}^\dag||^2+||\overline{T}^\dag||||T^\dag||
+ ||T^\dag||^2)||\delta T||.
$$
Combining it with $\|\overline{T}^\dag\|=K_{\overline{T}}\leq M$, we can conclude that
$\lim\limits_{\|\delta T\|\to 0}\overline{T}^\dagger=T^\dagger.$ To the end, we only need to prove $(8) \Rightarrow (5)$. If $\overline{T}$ has the bounded Moore--Penrose inverse $\overline{T}^\dagger$ with $\lim\limits_{\|\delta T\|\to 0}\overline{T}^\dagger=T^\dagger,$
then $P_{N(\overline{T})}^\perp=I-\overline{T}^\dag\overline{T}$ and $ P_{N(T)}^\perp=I-T^\dag T $
and $\lim\limits_{\|\delta T\|\to 0}P_{N(\overline{T})}^\perp=P_{N(T)}^\perp$.
Without loss of generality, we may assume $\|P_{N(\overline{T})}^\perp-P_{N(T)}^\perp\|<1$. Then
\begin{eqnarray*}
P_{N(T)}^\perp N(\overline{T})&=&(I-T^\dag T)N(\overline{T})\\
&=&(I-T^\dagger\overline{T}+T^\dagger\delta T)N(\overline{T})\\
&=&(I+T^\dagger\delta T)N(\overline{T})
 \end{eqnarray*} and by the Banach Lemma, the operator
$$W=I-P_{N(T)}^\perp+P_{N(\overline{T})}^\perp P_{N(T)}^\perp=I+(P_{N(\overline{T})}^\perp-P_{N(T)}^\perp)P_{N(T)}^\perp $$
is invertible and its inverse $W^{-1}: X\to X$ is bounded.
Take any $y\in
R(\overline{T})\cap N(T^\dagger)$, then $y=\overline{T}x$ and $
T^\dagger\overline{T}x=0,$ where $ x\in X$. Hence
$$T(I+T^\dagger \delta T)x=T(I+T^\dagger\overline{T}-T^\dagger T)x=0,$$
 which implies $(I+T^\dagger \delta T)x\in N(T).$
Therefore
\begin{eqnarray}
 (I+T^\dagger \delta T)x
&=&P_{N(T)}^\perp(I+T^\dagger \delta T)x\nonumber\\
&=&P_{N(T)}^\perp WW^{-1}(I+T^\dagger \delta T)x\nonumber\\
&=&P_{N(T)}^\perp(I-P_{N(T)}^\perp+P_{N(\overline{T})}^\perp P_{N(T)}^\perp)W^{-1}(I+T^\dagger \delta T)x\nonumber\\
&=&P_{N(T)}^\perp P_{N(\overline{T})}^\perp P_{N(T)}^\perp W^{-1}(I+T^\dagger \delta T)x\nonumber\\
&\in& P_{N(T)}^\perp N(\overline{T})=(I+T^\dagger\delta T)N(\overline{T}).\nonumber
 \end{eqnarray}
Since $T^\dagger\delta T+I$ is invertible, we get $y=\overline{T}x=0$. Thus we obtain (5).
\end{proof}

\begin{theorem}\label{The3.2} Let $X$, $Y$ be Hilbert spaces, $T\in B(X,Y)$ be a finite rank operator and $T^\dagger\in B(Y,X)$ be a
Moore--Penrose inverse of $T$. If $\delta T\in B(X,Y)$  satisfies
$\|\delta T\| \|T^\dagger\| <1$, then $\overline{T}=T+\delta T$ has
the Hyers--Ulam stability and the Hyers--Ulam stability constant
$K_{\overline{T}}$ satisfies $\lim\limits_{\|\delta T\|\to
0}K_{\overline{T}}=K_T$ if and only if
$${\rm Rank}\ \overline{T}={\rm Rank}\ T<+\infty.$$

\end{theorem}
\begin{proof} The necessity follows from $(6)\Rightarrow (2)$ in Theorem \ref{The3.1}. Next, we shall show the sufficiency. By Lemma \ref{Lem2.3}, $\overline{T}B$ is a projector from $Y$ onto $R(\overline{T}B)$. Then
$$R(T)\dot{+} N(T^\dagger)=Y=R(\overline{T}B)\oplus N(B)=R(\overline{T}B)\oplus N(T^\dagger).$$
Noting $\dim R(T)={\rm codim}N(T^\dagger)=\dim R(\overline{T}B)$, we have $\dim R(\overline{T})=\dim R(\overline{T}B)$, so $R(\overline{T}B)=R(\overline{T})$. Hence by $R(\overline{T}B)\cap N(T^\dagger)=R(\overline{T}B)\cap N(B)=\{0\}$, we get $R(\overline{T})\cap N(T^\dagger)=\{0\}$.\end{proof}
\begin{theorem}\label{The3.3} Let $X$, $Y$ be Hilbert spaces and let $T\in B(X,Y)$ be a semi-Fredholm operator. If $T$ has the bounded Moore--Penrose inverse $T^\dagger\in B(Y,X)$ and $\delta T\in
B(X,Y)$ satisfies $\|\delta T\|\|T^\dagger\|<1$, then $\overline{T}=T+\delta T$ has the Hyers--Ulam stability and the Hyers--Ulam stability constant $K_{\overline{T}}$
satisfies $\lim\limits_{\|\delta T\|\to 0}K_{\overline{T}}=K_T$ if and only if
 $$ either \ \dim N(\overline{T})=\dim N(T)<\infty\ or \ {\rm codim}R(\overline{T})={\rm codim}R(T)<\infty.$$
\end{theorem}
\begin{proof}
It is easy to see that the sufficiency follows from Theorem \ref{The2.3}. Next we shall prove the necessity. If $\dim N(T)<\infty$, then by  $(6)\Rightarrow (4)$ in Theorem \ref{The3.1}, we can see $\dim N(\overline{T})=\dim N(T)$. If ${\rm codim}R(T)<\infty$, then by Theorem \ref{The3.1}, $T^\dagger(I+\delta TT^\dagger)^{-1}:\ Y\to X$
 is a bounded generalized inverse of $T+\delta T$. Hence from Lemma \ref{Lem2.4}, $N(B)=N(T^\dagger)$ and $$Y=R(T)\dot{+} N(T^\dagger)=R(\overline{T})\oplus N(B)=R(\overline{T})\oplus N(T^\dagger).$$
 Thus $\dim N(T^\dagger)={\rm codim}R(T)<\infty$ and therefore ${\rm codim}R(\overline{T})={\rm codim}R(T).$
\end{proof}

\end{document}